\documentclass[letterpaper, 10 pt, conference]{ieeeconf}  

\IEEEoverridecommandlockouts                              
\usepackage{amsmath}
\usepackage{amssymb}
\usepackage{amsfonts}

\usepackage[ruled,vlined,linesnumbered,inoutnumbered]{algorithm2e}

\usepackage{graphics,graphicx,epsf,epstopdf,subfigure}
\graphicspath{{./Figures/}}
\ifpdf
\DeclareGraphicsExtensions{.eps,.pdf,.png,.jpg}
\else
\DeclareGraphicsExtensions{.eps}
\fi

\usepackage[breaklinks,colorlinks,linkcolor=black,citecolor=black,urlcolor=black]{hyperref}

\DeclareMathOperator*{\argmin}{arg\,min}

\newcommand{\bR}{\mathbb{R}}

\newcommand{\cO}{\mathcal{O}}

\newcommand{\Rnp}{\mathbb{R}^{n \times p}}

\newcommand{\Rnm}{\mathbb{R}^{n \times m}}

\newcommand{\zz}{^{\top}}
\newcommand{\ff}{_{\mathrm{F}}}
\newcommand{\fs}{^2_{\mathrm{F}}}

\newcommand{\dist}{\mathrm{dist}}

\newcommand{\tr}{\mathrm{tr}}
 
\newcommand{\sym}{\mathrm{sym}}
\newcommand{\proj}{\mathrm{proj}}
\newcommand{\env}{\mathrm{env}}
\newcommand{\prox}{\mathrm{prox}}

\newcommand{\bfone}{\mathbf{1}}

\newcommand{\Snp}{\mathcal{S}_{n,p}}

\newcommand{\dkh}[1]{\left(#1\right)}
\newcommand{\hkh}[1]{\left\{#1\right\}}

\newcommand{\jkh}[1]{\left\langle#1\right\rangle}
\newcommand{\norm}[1]{\left\|#1\right\|}
\newcommand{\abs}[1]{\left\lvert #1\right\rvert}

\newcommand{\iid}{i \in [d]}

\newcommand{\sumiid}{\sum_{i=1}^d}

\newcommand{\sumjjd}{\sum_{j=1}^d}

\newcommand{\Xik}{X_{i}^{(k)}}
\newcommand{\Xjk}{X_{j}^{(k)}}
\newcommand{\Xikn}{X_{i}^{(k + 1)}}

\newcommand{\Dik}{D_{i}^{(k)}}
\newcommand{\Djk}{D_{j}^{(k)}}
\newcommand{\Dikn}{D_{i}^{(k + 1)}}

\newcommand{\Hik}{H_{i}^{(k)}}

\newcommand{\Hikn}{H_{i}^{(k + 1)}}

\newtheorem{theorem}{Theorem}

\newtheorem{lemma}[theorem]{Lemma}
\newtheorem{proposition}[theorem]{Proposition}
\newtheorem{definition}{Definition}

\newtheorem{assumption}{Assumption}
\newtheorem{condition}{Condition}

\title{\LARGE \bf
Smoothing Gradient Tracking 
for Decentralized Optimization over the Stiefel Manifold 
with Non-smooth Regularizers
}

\author{Lei Wang and Xin Liu
\thanks{The first author is supported by 
	the National Key R\&D Program of China (No. 2020YFA0711900, 2020YFA0711904).
	The second author is supported in part by 
	the National Natural Science Foundation of China 
	(No. 12125108, 11971466, 12288201, 12021001, 11991021)
	and Key Research Program of Frontier Sciences, 
	Chinese Academy of Sciences (No. ZDBS-LY-7022).}
\thanks{L. Wang and X. Liu are both with 
	State Key Laboratory of Scientific and Engineering Computing,
	Academy of Mathematics and Systems Science, 
	Chinese Academy of Sciences, Beijing, China,
	and University of Chinese Academy of Sciences, Beijing, China.
    {\tt\small wlkings@lsec.cc.ac.cn, liuxin@lsec.cc.ac.cn}}%
}

\begin{document}

\maketitle
\thispagestyle{empty}
\pagestyle{empty}

\begin{abstract}

Recently, decentralized optimization over the Stiefel manifold
has attacked tremendous attentions 
due to its wide range of applications in various fields.
Existing methods rely on the gradients to update variables,
which are not applicable to the objective functions
with non-smooth regularizers,
such as sparse PCA.
In this paper, to the best of our knowledge,
we propose the first decentralized algorithm 
for non-smooth optimization over Stiefel manifolds.
Our algorithm approximates the non-smooth part of objective function 
by its Moreau envelope,
and then existing algorithms for smooth optimization can be deployed.
We establish the convergence guarantee 
with the iteration complexity of $\cO (\epsilon^{-4})$.
Numerical experiments conducted under the decentralized setting
demonstrate the effectiveness and efficiency of our algorithm.

\end{abstract}

\section{INTRODUCTION}

Given a set of $d$ agents connected by a communication network,
we focus on the optimization problem
over the Stiefel manifold $\Snp := \{X \in \Rnp \mid X\zz X = I_p\}$
with non-smooth regularizers of the following form:
\begin{equation}\label{opt:stiefel}
	\begin{aligned}
		\min\limits_{X \in \Snp} \hspace{2mm} 
		& \sumiid (f_i(X) + g_i (X)), \\
	\end{aligned}
\end{equation}
where $f_i: \Rnp \to \bR$ and $g_i: \Rnp \to \bR$ are two local functions
privately owned by agent $i \in [d] := \{1, \dotsc, d\}$,
and $I_p$ denotes the $p \times p$ identity matrix with $p \leq n$.
We consider the scenario that 
the agents can only exchange information 
with their immediate neighbors through the network,
which can be modeled as a connected undirected graph.
Under this decentralized setting,
there is not a center to aggregate the local information 
and coordinate the optimization process.
Consequently, each agent has to maintain a local variable $X_i$ 
as a copy of the common variable $X$.
The goal of decentralized optimization is 
to seek a global consensus
such that each local variable is a solution to problem \eqref{opt:stiefel}
through local communication.

Throughout this paper, we make the following assumptions 
about problem \eqref{opt:stiefel}.

\begin{assumption}\label{asp:objective}
	
	The functions $f_i$ and $g_i$ satisfy the following conditions for any $\iid$.
	
	\begin{enumerate}
		
		\item $f_i$ is first-order differentiable 
		and its Euclidean gradient $\nabla f_i$ 
		is Lipschitz continuous over $\Snp$
		with the corresponding Lipschitz constant $L_{f_i} \geq 0$.

		\item $g_i$ is convex and Lipschitz continuous
		with the corresponding Lipschitz constant $L_{g_i} \geq 0$.
		
	\end{enumerate}
	
\end{assumption}
For convenience, we denote $L_f := \max_{\iid} L_{f_i}$
and $L_g := \max_{\iid} L_{g_i}$.


By virtue of its versatility, problem \eqref{opt:stiefel} arises naturally 
in many scientific and engineering applications,
such as sparse principal component analysis (PCA) \cite{Jolliffe2003,Wang2021communication},
deep neural networks with orthogonality constraints \cite{Arjovsky2016,Huang2018},
dual principal component analysis \cite{Tsakiris2018dual,Zhu2018dual},
and dictionary learning \cite{Zhu2019linearly,Lu2022decentralized}.
However, under the decentralized setting,
it is quite changing to solve problem~\eqref{opt:stiefel}.
The difficulty lies primarily in the non-smoothness of objective function
and the non-convexity of manifold constraint.


\subsection{Related Works}
\label{subsec:review}

Recent years have seen the extensive development of 
decentralized optimization over Stiefel manifolds.
Existing algorithms can be divided into two categories.
The first category leverages the geometric tools from Riemannian optimization \cite{Absil2008}
to solve this problem,
including DRGTA \cite{Chen2021decentralized}
and DRNGD \cite{Hu2023decentralized}.
These algorithms directly seek a consensus on Stiefel manifolds \cite{Chen2021local},
which require multiple rounds of communications to guarantee the convergence.
As a result, this communication bottleneck hinders the scalability in large-scale networks.
The second category, built on a different framework,
constructs exact penalty models for optimization over Stiefel manifolds,
which are then solved by unconstrained decentralized algorithms.
Therefore, this category attempts to reach a consensus
in the ambient Euclidean space alternatively.
Two members of this category are DESTINY \cite{Wang2022decentralized}
and VRSGT \cite{Wang2023variance}.
These algorithms only invoke a single round of communications per iteration,
which can provide a high degree of communication-efficiency in general.

We emphasize that the above mentioned methods are tailored for 
smooth optimization problems over Stiefel manifolds,
since the gradients of objective function are computed per iteration.
To the best of our knowledge, there is no decentralized algorithm 
that can solve the non-smooth problem \eqref{opt:stiefel}.

It is worthy of mentioning that smoothing methods have been introduced in 
Riemannian optimization to solve the non-smooth problems.
For example, \cite{Zhang2021riemannian} extends the smoothing steepest descent method 
from Euclidean spaces to Riemannian manifolds.
Moreover, \cite{Peng2022riemannian} and  \cite{Zhu2023smoothing} 
propose a family of Riemannian gradient type methods 
based on the smooth approximation of objective functions.
Generally speaking, these algorithms require some global information
that is not available under the decentralized setting.
In addition, a Riemannian ADMM algorithm 
is developed in \cite{Li2022riemannian}
to solve the smoothed problem
with a favorable numerical performance.
The convergence is not guaranteed 
with the additional consensus constraint
under the decentralized setting.
In summary, the above-mentioned algorithms are 
tailored for centralized optimization problems,
which can not be straightforwardly extended to the decentralized setting.


\subsection{Contributions}

In this paper, we propose the first decentralized algorithm
for the optimization problem \eqref{opt:stiefel}
over the Stiefel manifold with non-smooth regularizers. 
The smoothing technique tides us over the obstacle 
to handling the combination of non-smoothness and non-convexity.
Our algorithm attempts to solve the smoothed proxy of problem \eqref{opt:stiefel},
where the non-smooth regularizers are replaced by their Moreau envelopes.
Even under the centralized setting, 
our algorithm provides a novel alternative for 
the non-smooth optimization problem over the Stiefel manifold.

We establish the global convergence of our algorithm 
to a first-order $\epsilon$-stationary point in $\cO (\epsilon^{-4})$ iterations.
Such theoretical guarantee matches the complexities of centralized approaches 
to non-smooth optimization over Stiefel manifolds,
such as Riemannian ADMM algorithm \cite{Li2022riemannian}
and Riemannian subgradient-type method \cite{Li2021weakly}.
Preliminary numerical experiments validate the effectiveness of our smoothing technique.
Moreover, our algorithm has a promising performance in sparse PCA problems.

\subsection{Notations}

The Euclidean inner product of two matrices $Y_1, Y_2$ 
with the same size is defined as $\jkh{Y_1, Y_2}=\tr(Y_1\zz Y_2)$,
where $\tr (B)$ stands for the trace of a square matrix $B$.
And the notation $\sym (B) = (B + B\zz) / 2$ represents 
the symmetric part of $B$.
The Frobenius norm and 2-norm of a given matrix $C$ 
are denoted by $\norm{C}\ff$ and $\norm{C}_2$, respectively. 
The $(i, j)$-th entry of a matrix $C$ is represented by $C (i, j)$.
Given a differentiable function $f (X) : \Rnp \to \bR$, 
the Euclidean gradient of $f$ with respect to $X$ is represented by $\nabla f (X)$.

\section{PRELIMINARIES}

This section introduces several preliminaries of our algorithm.

\subsection{Stationarity Condition}


We first introduce the definition of Clarke subgradient \cite{Clarke1990} 
for non-smooth functions.

\begin{definition}
	Suppose $f: \Rnp \to \bR$ is a Lipschitz continuous function.
	The generalized directional derivative of $f$ at the point $X \in \Rnp$
	along the direction $H \in  \Rnp$ is defined by:
	\begin{equation*}
		f^{\circ} (X; H) := \limsup\limits_{Y \to X,\, t \to 0^+} \dfrac{f (Y + t H) - f(Y)}{t}.
	\end{equation*}
	Based on generalized directional derivative of $f$,
	the (Clark) subgradient of $f$ is defined by:
	\begin{equation*}
		\partial f(X) := \{G \in \Rnp \mid \jkh{G, H} \leq f^{\circ} (X; H) \}.
	\end{equation*}
\end{definition}

As discussed in \cite{Yang2014} and \cite{Chen2020},
the first-order stationarity condition of \eqref{opt:stiefel} can be stated as follows.

\begin{definition} 
	\label{def:stationary}
	A point $X \in \Rnp$ is called a first-order stationary point 
	of \eqref{opt:stiefel} if it satisfies the following conditions.
	\begin{equation*} 
		\left\{
		\begin{aligned}
			& 0 \in \proj_{X} \dkh{ \sumiid \dkh{\nabla f_i (X) + \partial g_i (X)} }, \\
			& X\zz X = I_p,
		\end{aligned}
		\right.
	\end{equation*}
	where $\proj_{X} (Y) := Y - X \sym (X\zz Y)$.
\end{definition}


For a point $X \in \Snp$, $\proj_{X} (\cdot)$ is nothing but the orthogonal projection 
onto the tangent space of $\Snp$ \cite{Absil2008}.
Based on Definition \ref{def:stationary}, 
we define the following notion of 
first-order $\epsilon$-stationary point.

\begin{definition} 
	\label{def:epsilon-stationary}
	A point $X \in \Rnp$ is called a first-order 
	$\epsilon$-stationary point of \eqref{opt:stiefel} 
	if there exists $\{Y_i \in \Rnp\}_{i = 1}^d$ such that the following conditions hold.
	\begin{equation*}
		\left\{
		\begin{aligned}
			& \dist\dkh{0, \proj_{X} \dkh{ 
					\sumiid \dkh{\nabla f_i (X) + \partial g_i (Y_i)} } 
			} \leq \epsilon, \\
			& \norm{X - Y_i}\ff \leq \epsilon, \quad \iid, \\
			& \norm{X\zz X - I_p}\ff \leq \epsilon.
		\end{aligned}
		\right.
	\end{equation*}
\end{definition}

One can readily check that a first-order $\epsilon$-stationary point 
will reduce to a first-order stationary point if $\epsilon = 0$.

\subsection{Mixing Matrix}

In the context of decentralized optimization,
we usually associate the network with a mixing matrix 
denoted by $W = [W(i, j)] \in \bR^{d \times d}$
to conform to the underlying communication structure.

\begin{assumption}\label{asp:network}
	
	The mixing matrix $W \in \bR^{d \times d}$ satisfies the following conditions.
	
	\begin{enumerate}
		
		\item $W$ is symmetric.
		
		\item $W$ is doubly stochastic, namely, $W$ is nonnegative 
		and $W \mathbf{1}_d = W\zz \mathbf{1}_d = \mathbf{1}_d$,
		where $\bfone_d \in \bR^d$ stands for the $d$-dimensional vector of all ones.
		
		\item $W(i, j) = 0$ if $i$ and $j$ are not connected and $i \neq j$.
		
	\end{enumerate}
	
\end{assumption}
The mixing matrix $W$ in Assumption \ref{asp:network},
which is standard in the literature,
always exists and can be constructed efficiently via exchange of local degree information 
between the agents.
We refer interested readers to \cite{Yuan2016,Shi2015,Nedic2018network} for more details.
According to the Perron-Frobenius Theorem \cite{Pillai2005perron},
we know that the eigenvalues of $W$ lie in $[-1, 1]$ and
\begin{equation*}
	\lambda := \norm{W - \bfone_d \bfone_d\zz / d}_2 < 1.
\end{equation*}
The parameter $\lambda$ measures the connectedness of networks.

\section{SMOOTHING TECHNIQUE}

Based on the smoothing technique,
we propose a novel decentralized algorithm 
to solve the optimization problem \eqref{opt:stiefel}
with non-smooth regularizers.

\subsection{Moreau Envelope}

Under the decentralized setting,
the combination of  non-smoothness and non-convexity
makes it intractable to tackle the problem \eqref{opt:stiefel}.
If there is only one of them, 
this problem is relatively easier to solve.
This motivates us to replace the non-smooth part of objective function 
by its Moreau envelope \cite{Moreau1965proximite,Rockafellar2009variational}
as a smooth approximation.
Then we can take advantage of existing algorithms for smooth problems 
to solve problem \eqref{opt:stiefel}.
This kind of algorithm is usually called smoothing algorithm \cite{Chen2012smoothing}.
The Moreau envelope and the closely related proximal operator are defined as follows.

%


\begin{definition} \label{def:moreau}
	For a proper, convex and lower semi-continuous function $g: \Rnp \to \bR$, 
	the Moreau envelope of $g$ with the smoothing parameter $\sigma > 0$ is given by
	\begin{equation}
		\label{eq:env}
		\env_{\sigma, g} (X) := \min_{Y \in \Rnp} 
		\hkh{ g (Y) + \dfrac{1}{2 \sigma} \norm{Y - X}\fs}.
	\end{equation}
	And the proximal operator of $g$ is the global minimizer 
	of the above optimization problem, that is,
	\begin{equation}
		\label{eq:prox}
		\prox_{\sigma, g} (X) 
		:= \argmin_{Y \in \Rnp} 
		\hkh{ g (Y) + \dfrac{1}{2 \sigma} \norm{Y - X}\fs}.
	\end{equation}
\end{definition}

The following proposition indicates that the Moreau envelope $\env_{\sigma, g} (X)$
can be used to approximate the non-smooth function $g$,
and the approximation error is controlled by the smoothing parameter $\sigma$.

\begin{proposition}[\cite{Bohm2021variable}]
	\label{prop:moreau-approx}
	 Let $g: \Rnp \to \bR$ be a proper, convex and lower semi-continuous function.
	Suppose $g$ is Lipschitz continuous with the corresponding Lipschitz constant $L \geq 0$.
	Then for any $\sigma > 0$, it holds that
	\begin{equation*}
		\env_{\sigma, g} (X) \leq g (X) 
		\leq \env_{\sigma, g} (X) 
		+ \dfrac{1}{2} \sigma L^2.
	\end{equation*}
\end{proposition}

Furthermore, the Moreau envelope $\env_{\sigma, g} (X)$ is a smooth function 
with the parameter $\sigma$ controlling the amount of smoothness.

\begin{proposition}[\cite{Bohm2021variable}]
	\label{prop:moreau-smooth}
	Let $g: \Rnp \to \bR$ be a proper, convex and lower semi-continuous function.
	Suppose $g$ is Lipschitz continuous with the corresponding Lipschitz constant $L \geq 0$.
	Then the Moreau envelope $\env_{\sigma, g} (X)$ is first-order continuously differentiable,
	and its Euclidean gradient has the following form:
	\begin{equation*}
		\nabla \env_{\sigma, g} (X)
		= \dfrac{1}{\sigma} (X - \prox_{\sigma, g} (X)).
	\end{equation*}
	Moreover, for any $X \in \Rnp$, we have 
	\begin{equation*}
		\norm{\nabla \env_{\sigma, g} (X)}\ff 
		\leq L.
	\end{equation*}
	Finally, $\nabla \env_{\sigma, g} (X)$ is Lipschitz continuous
	with the corresponding Lipschitz constant $1 / \sigma$.
\end{proposition}

\subsection{Smoothed Problem}

Based on Proposition \ref{prop:moreau-approx} and Proposition \ref{prop:moreau-smooth},
the Moreau envelope offers a smooth approximation to non-smooth functions. 
By resorting to this powerful tool, 
we can obtain the following smoothed problem of \eqref{opt:stiefel}.
\begin{equation}\label{opt:stiefel-s}
	\begin{aligned}
		\min\limits_{X \in \Snp} \hspace{2mm} 
		& \sumiid h_i (X), \\
	\end{aligned}
\end{equation}
where $h_i (X) := f_i(X) + \env_{\sigma, g_i} (X)$
is a local function privately held by agent $i$.

According to the discussions in \cite{Wang2021multipliers},
a point $X \in \Rnp$
satisfies the first-order $\epsilon$-stationarity condition 
of problem \eqref{opt:stiefel-s} 
if and only if 
\begin{equation*}
	\left\{
	\begin{aligned}
		& \norm{\proj_{X} \dkh{ G (X) } 
		}\ff \leq \epsilon, \\
		& \norm{X\zz X - I_p}\ff \leq \epsilon,
	\end{aligned}
	\right.
\end{equation*}
where $G (X) = \sumiid G_i (X)$ with
\begin{equation*}
	G_i (X) := \nabla h_i (X) 
	= \nabla f_i (X) + \nabla \env_{\sigma, g_i} (X).
\end{equation*}
We have the following lemma.

\begin{lemma}
	\label{le:stationary}
	Suppose $X \in \Rnp$ is a first-order $\epsilon$-stationary point 
	of the smoothed problem \eqref{opt:stiefel-s} with
	\begin{equation*}
		0 < \sigma \leq \dfrac{\epsilon}{2 L_g}.
	\end{equation*}
	Then $X$ is also a first-order $\epsilon$-stationary point 
	of problem~\eqref{opt:stiefel}.
\end{lemma}

\begin{proof}
	Let $Y_i = \prox_{\sigma, g_i} (X)$.
	Then it follows from the optimality condition of \eqref{eq:env} that
	\begin{equation*}
		0 \in \partial g_i (Y_i) + \dfrac{1}{\sigma} (Y_i - X),
	\end{equation*}
	which further yields that
	\begin{equation*}
		\nabla \env_{\sigma, g_i} (X)
		= \dfrac{1}{\sigma} (X - Y_i) 
		\in \partial g_i (Y_i).
	\end{equation*}
	Hence, we can obtain that
	\begin{equation*}
		\begin{aligned}
			& \dist\dkh{0, \proj_{X} \dkh{ 
					\sumiid \dkh{\nabla f_i (X) + \partial g_i (Y_i)} } 
			} \\
			& \leq \norm{\proj_{X} \dkh{ G (X) } 
			}\ff
			\leq \epsilon. \\
		\end{aligned}
	\end{equation*}
	In addition, according to the definition of proximal operator, 
	we have
	\begin{equation*}
		g_i (Y_i) + \dfrac{1}{2 \sigma} \norm{Y_i - X}\fs
		\leq g_i (X).
	\end{equation*}
	This together with the Lipschitz continuity of $g_i$ that
	\begin{equation*}
		\dfrac{1}{2 \sigma} \norm{X - Y_i}\fs
		\leq g_i (X) - g_i (Y)
		\leq L_{g_i} \norm{X - Y_i}\ff,
	\end{equation*}
	which implies that $\norm{X - Y_i}\ff 
	\leq 2 \sigma L_{g_i} \leq \epsilon$.
	According to Definition \ref{def:epsilon-stationary},
	we know that $X$ is a first-order $\epsilon$-stationary point
	of problem \eqref{opt:stiefel}.
	The proof is completed.
\end{proof}

Lemma \ref{le:stationary} guarantees that
one can always find an approximate first-order stationary point of \eqref{opt:stiefel}
by solving the smoothed problem \eqref{opt:stiefel-s}.

\subsection{Algorithm Development}

In this subsection, we intend to solve the smoothed problem \eqref{opt:stiefel-s}.
Among existing algorithms introduced in Subsection \ref{subsec:review}, 
DESTINY \cite{Wang2022decentralized} is chosen
due to its communication-efficiency.

Let $\Xik$ and $\Dik$ denote the $k$-th iterate of 
local variable and gradient tracker at agent $i$, respectively.
In our algorithm, the local variable is first updated
by performing a descent step along the direction of $\Dik$
and communicating with neighbors,
that is,
\begin{equation}
	\label{eq:update-x}
	\Xikn := \sumjjd W(i, j) \dkh{\Xjk - \eta \Djk},
\end{equation}
where $\eta > 0$ is the stepsize.
Then, the local descent direction $\Hikn$ can be evaluated as follows.
\begin{equation}
	\label{eq:update-h}
	\begin{aligned}
		\Hikn := {} & \beta \Xikn \dkh{(\Xikn)\zz \Xikn - I_p} \\
		& + R_i (\Xikn),
	\end{aligned}
\end{equation}
where $\beta > 0$ is a penalty parameter and
\begin{equation*}
	\begin{aligned}
		R_i (X) := \dfrac{1}{2} G_i (X) \dkh{3 I_p - X\zz X}
		- X \sym \dkh{X\zz G_i (X)}.
	\end{aligned}
\end{equation*}
For more details about the construction of $\Hikn$, 
we refer interested readers to \cite{Wang2022decentralized}. 
Finally, each agent $i$ updates $\Dikn$
based on the following gradient tracking technique.
\begin{equation}
	\label{eq:update-d}
	\Dikn := \sumjjd W(i, j) \Djk 
	+ \Hikn - \Hik.
\end{equation}

We formally present the detailed algorithmic framework as Algorithm \ref{alg:THANOS},
named {\it ``decen\underline{t}ralized smoot\underline{h}ing 
gr\underline{a}dient 
tracki\underline{n}g 
\underline{o}ver
\underline{S}tiefel manifolds''} 
and abbreviated to THANOS.
In principle, one can devise an adaptive strategy 
to update the smoothing parameter
based on the global objective function value, 
such as \cite{Chen2012smoothing,Bian2020smoothing,Liu2022linearly}.
Such information is not available under the decentralized setting.
Therefore, the smoothing parameter is fixed in THANOS.

\IncMargin{0.5em}
\begin{algorithm}[ht!]
	\caption{Decentralized smoothing gradient tracking over Stiefel manifolds (THANOS).} 
	\label{alg:THANOS}
	
	\KwIn{initial guess $X_{\mathrm{initial}} \in \Snp$, stepsize $\eta > 0$, 
		smoothing parameter $\sigma > 0$,
		and penalty parameter $\beta > 0$.}
	
	Set $k := 0$.
	
	For any $\iid$, initialize $\Xik := X_{\mathrm{initial}}$ and $\Dik := \Hik$. 
	
	\While{``not converged''}
	{
	
		\For{all $\iid$ in parallel}
		{
			
			Update $\Xikn$ by \eqref{eq:update-x}.
			
			Compute $\Hikn$ by \eqref{eq:update-h}.
			
			Update $\Dikn$ by \eqref{eq:update-d}.
			
			Set $k := k + 1$.
			
		}
		
	}
	
	\KwOut{$\{\Xik\}$.}
	
\end{algorithm}

\addtolength{\textheight}{-3cm}   

\section{CONVERGENCE ANALYSIS}

This section is devoted to the convergence analysis of THANOS.
Towards this end, we need to impose several mild conditions on $\beta$  and $\eta$,
which are stated below to facilitate the narrative.
%
%
%

\begin{condition}
	\label{cond:parameter}
	(i) The penalty parameter $\beta$ satisfies 
	\begin{equation*}
		\begin{aligned}
			\beta > \max\left\{
				\dfrac{6 + 21 (M_f + L_g)}{5}, \,
				\dfrac{72 (4 + 3 M_g)}{5}, \,
				\dfrac{1}{7dp + 6d}, \,
				\right. \\
				\left.
				22 \dkh{L_f + \dfrac{1}{\sigma}}^2
				\right\},
		\end{aligned}
	\end{equation*}
	where $M_f := \sup_{\iid}\, \{ \norm{\nabla f_i (X)}\ff 
	\mid \norm{X}\ff \leq \sqrt{7dp / 6} + \sqrt{d} \}$
	and $M_g := 3 (M_f + L_g) (7dp + 6d + 3) / 6$
	are two positive constants.
	
	(ii) The stepsize $\eta$ satisfies
	\begin{equation*}
		0 < \eta < \dfrac{d (1 - \lambda^2)}{48 (L_r + (7dp + 6d) \beta)^2},
	\end{equation*}
	where $L_r  := 7dp (L_f + 1 / \sigma) + 6d + 3$ is a positive constant.
\end{condition}

\begin{proposition}
	\label{prop:rate}
	Suppose Assumption \ref{asp:objective} and Assumption \ref{asp:network} hold,
	and $\{\bar{X}^{(k)}\}$  is the average  sequence of local iterates 
	generated by Algorithm \ref{alg:THANOS},
	where $\bar{X}^{(k)} := \sumiid \Xik / d$.
	Let the algorithmic parameters $\beta$ and $\eta$ satisfy Condition \ref{cond:parameter}.
	Then $\{\bar{X}^{(k)}\}$ has at least one accumulation point,
	and any accumulation point $\bar{X}^{\ast}$ 
	is a first-order stationary point of the smoothed problem \eqref{opt:stiefel-s}.
	Moreover, the following relationships hold.
	\begin{equation*}
		\min_{k = 0, 1, \dotsc, K - 1} \norm{R (\bar{X}^{(k)})}\fs
		\leq \dfrac{2 C}{\eta K},
	\end{equation*}
	and
	\begin{equation*}
		\min_{k = 0, 1, \dotsc, K - 1} 
		\norm{(\bar{X}^{(k)})\zz \bar{X}^{(k)} - I_p}\fs
		\leq \dfrac{2 C}{\eta (L_f + 1 / \sigma)^2 K},
	\end{equation*}
	where $R (X) := \sumiid R_i (X)$ 
	and $C > 0$ is a constant independent of $\sigma$.
\end{proposition}

\begin{proof}
	It follows from Proposition \ref{prop:moreau-smooth} that
	the local function $h_i$ in \eqref{opt:stiefel-s} is first-order differentiable.
	Moreover, $\nabla h_i$ is Lipschitz continuous over $\Snp$,
	and the corresponding Lipschitz constant is $L_{f_i} + 1 / \sigma$.
	Then according to Theorem 10 in \cite{Wang2022decentralized},
	we can obtain the assertions of this proposition.
	The proof is completed.
\end{proof}

\begin{theorem}
	Suppose all the conditions in Proposition \ref{prop:rate} hold
	and 
	\begin{equation*}
		0 < \sigma \leq \dfrac{\epsilon}{2 L_g}.
	\end{equation*}
	Then Algorithm \ref{alg:THANOS} will return a first-order $\epsilon$-stationary point 
	of problem \eqref{opt:stiefel} in at most $\cO (\epsilon^{-4})$ iterations.
\end{theorem}

\begin{proof}
	By straightforward calculations, we have
	\begin{equation*}
		\begin{aligned}
			& \norm{\proj_{\bar{X}^{(k)}}
			\dkh{G (\bar{X}^{(k)})}}\ff \\
			& \leq \norm{R (\bar{X}^{(k)})}\ff 
			+ \dfrac{1}{2} \norm{G (\bar{X}^{(k)})}\ff
			\norm{(\bar{X}^{(k)})\zz \bar{X}^{(k)} - I_p}\ff \\
			& \leq \norm{R (\bar{X}^{(k)})}\ff 
			+ \dfrac{1}{2} (M_f + L_g)
			\norm{(\bar{X}^{(k)})\zz \bar{X}^{(k)} - I_p}\ff.
		\end{aligned}
	\end{equation*}
	Then it can be readily verified that
	\begin{equation*}
		\begin{aligned}
			& \norm{\proj_{\bar{X}^{(k)}}
				\dkh{G (\bar{X}^{(k)})}}\fs \\
			& \leq 2 \norm{R (\bar{X}^{(k)})}\fs 
			+ \dfrac{1}{2} (M_f + L_g)^2
			\norm{(\bar{X}^{(k)})\zz \bar{X}^{(k)} - I_p}\fs,
		\end{aligned}
	\end{equation*}
	which implies that
	\begin{equation*}
		\begin{aligned}
			& \min_{k = 0, 1, \dotsc, K - 1} 
			\norm{\proj_{\bar{X}^{(k)}}
				\dkh{G (\bar{X}^{(k)})}}\fs \\
			& \leq \dfrac{4 (L_f + 1 / \sigma)^2 C + (M_f + L_g)^2 C}{\eta (L_f + 1 / \sigma)^2 K}.
		\end{aligned}
	\end{equation*}
	According to Lemma \ref{le:stationary}, 
	Algorithm \ref{alg:THANOS} is guaranteed to 
	find a first-order $\epsilon$-stationary point if
	\begin{equation*}
		\left\{
		\begin{aligned}
			& \dfrac{4 (L_f + 1 / \sigma)^2 C + (M_f + L_g)^2 C}{\eta (L_f + 1 / \sigma)^2 K} 
			\leq \epsilon^2, \\
			& \dfrac{2 C}{\eta (L_f + 1 / \sigma)^2 K}
			\leq \epsilon^2,
		\end{aligned}
		\right.
	\end{equation*}
	namely,
	\begin{equation*}
		\begin{aligned}
		K \geq \max & \left\{
		\dfrac{4 (L_f + 1 / \sigma)^2 C + (M_f + L_g)^2 C}{\eta (L_f + 1 / \sigma)^2 \epsilon^2},
		\right. \\
	 	& \quad \left.
		\dfrac{2 C}{\eta (L_f + 1 / \sigma)^2 \epsilon^2}
		\right\}
		= \cO \dkh{\dfrac{1}{\epsilon^4}}.
		\end{aligned}
	\end{equation*}
	The proof is completed.
\end{proof}

\section{NUMERICAL EXPERIMENTS}

Comprehensive numerical experiments
are conducted in this section
to evaluate the numerical performance of THANOS.
We use the $\mathbf{Python}$ language
to implement the tested algorithms
with the communication realized via the package $\mathbf{mpi4py}$.
And the corresponding experiments are performed on a workstation
with two Intel Xeon Gold 6242R CPU processors (at $3.10$GHz$\times 20 \times 2$) 
and 510GB of RAM under Ubuntu 20.04.

\subsection{Test Problem}

In the numerical experiments, we test the performance of THANOS 
on the following sparse PCA problems.
\begin{equation}
	\label{opt:spca}
	\begin{aligned}
		\min\limits_{X \in \Snp} \hspace{2mm} 
		& -\frac{1}{2} \sumiid \tr \dkh{ X\zz A_i A_i\zz X} + \mu r (X), \\
	\end{aligned}
\end{equation}
where $A_i \in \bR^{n \times m_i}$ is the local data matrix 
privately owned by agent $i \in [d]$
that consists of $m_i$ samples with $n$ features,
the non-smooth regularizer $r (X)$ is imposed to 
promote specific sparsity structures in $X$,
and $\mu > 0$ is the parameter used to control the amount of sparseness.
We use $A = [A_1 \; A_2 \; \dotsb \; A_d] \in \Rnm$ to denote the global data matrix
such that each agent possesses a subset of samples,
where $m = m_1 + m_2 + \dotsb + m_d$.
This is a natural setting under the distributed circumstance \cite{Wang2020seeking}.
One can readily verify that \eqref{opt:spca} is a special case of \eqref{opt:stiefel}
by identifying $f_i (X) = - \tr (X\zz A_i A_i\zz X) / 2$
and $g_i (X) = \mu r (X) / d$
for any $\iid$.

We consider two different regularizers.
The first one is $\ell_l$-norm regularizer
\cite{Jolliffe2003}:
\begin{equation}
	\label{eq:l1}
	r (X) = \norm{X}_1 
	:= \sum_{i = 1}^n \sum_{j = 1}^p \abs{X(i, j)}.
\end{equation}
The second one is $\ell_{2,1}$-norm regularizer
\cite{Xiao2021exact}:
\begin{equation}
	\label{eq:l21}
	r (X) = \norm{X}_{2, 1} 
	:= \sum_{i = 1}^n \norm{X(i, \cdot)}_2,
\end{equation}
where $X(i, \cdot)$ denotes the $i$-th row of $X$.


\subsection{Numerical Results}


In the following experiments, 
we randomly generate the test matrix $A$
with $n = 10$ and $m = 320$.
The columns of $A$ are uniformly distributed into $d = 32$ agents.
Other parameters in problem \eqref{opt:spca}
as set as $p = 3$ and $\mu = 0.1$.
We construct an Erdos-Renyi network, 
where two agents are connected with a fixed probability $0.5$.
This network is associated with the Metropolis constant matrix \cite{Shi2015} 
as the mixing matrix $W$.

After the construction of $A$,
we employ the SLPG \cite{Xiao2021penalty} algorithm
to generate a high-precision solution $X^{\ast} \in \Snp$ to problem \eqref{opt:spca}
under the centralized environment.
Then we test the performance of THANOS on problem \eqref{opt:spca} 
for different values of smoothing parameter $\sigma$
with fixed penalty parameter $\beta = 1$.
We use the BB stepsize proposed in \cite{Wang2022decentralized} 
to accelerate the convergence.
The initial point $X_i^{(0)}$ is constructed from the leading $p$ left singular vectors of $A$,
which can be computed efficiently by DESTINY \cite{Wang2022decentralized} 
under the decentralized setting.

In each iteration of THANOS, 
we compute and record the error term defined by
\begin{equation*}
	\dist^{(k)} := \dfrac{1}{d} \sumiid \norm{\Xik - X^{\ast}}\ff,
\end{equation*}
and the feasibility violation defined by
\begin{equation*}
	\mathrm{feas}^{(k)} := \dfrac{1}{d} \sumiid \norm{(\Xik)\zz \Xik - I_p}\ff,
\end{equation*}
as the performance measurements.

\begin{figure}[t!]
	\centering
	
	\subfigure[$\dist^{(k)}$]{
		\label{subfig:dist_l1}
		\includegraphics[width=0.8\linewidth]{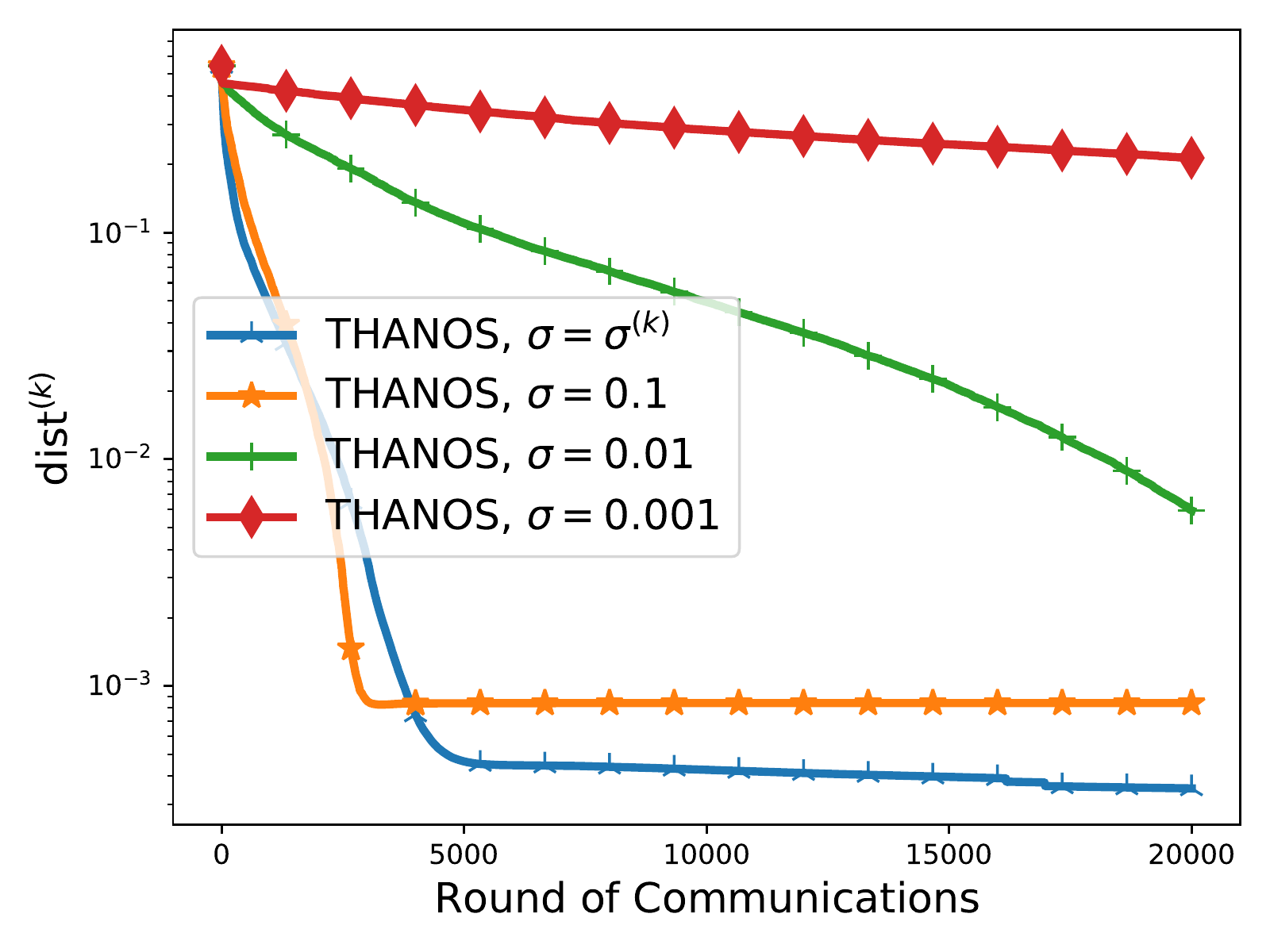}
	}
	
	\subfigure[$\mathrm{feas}^{(k)}$]{
		\label{subfig:feas_l1}
		\includegraphics[width=0.8\linewidth]{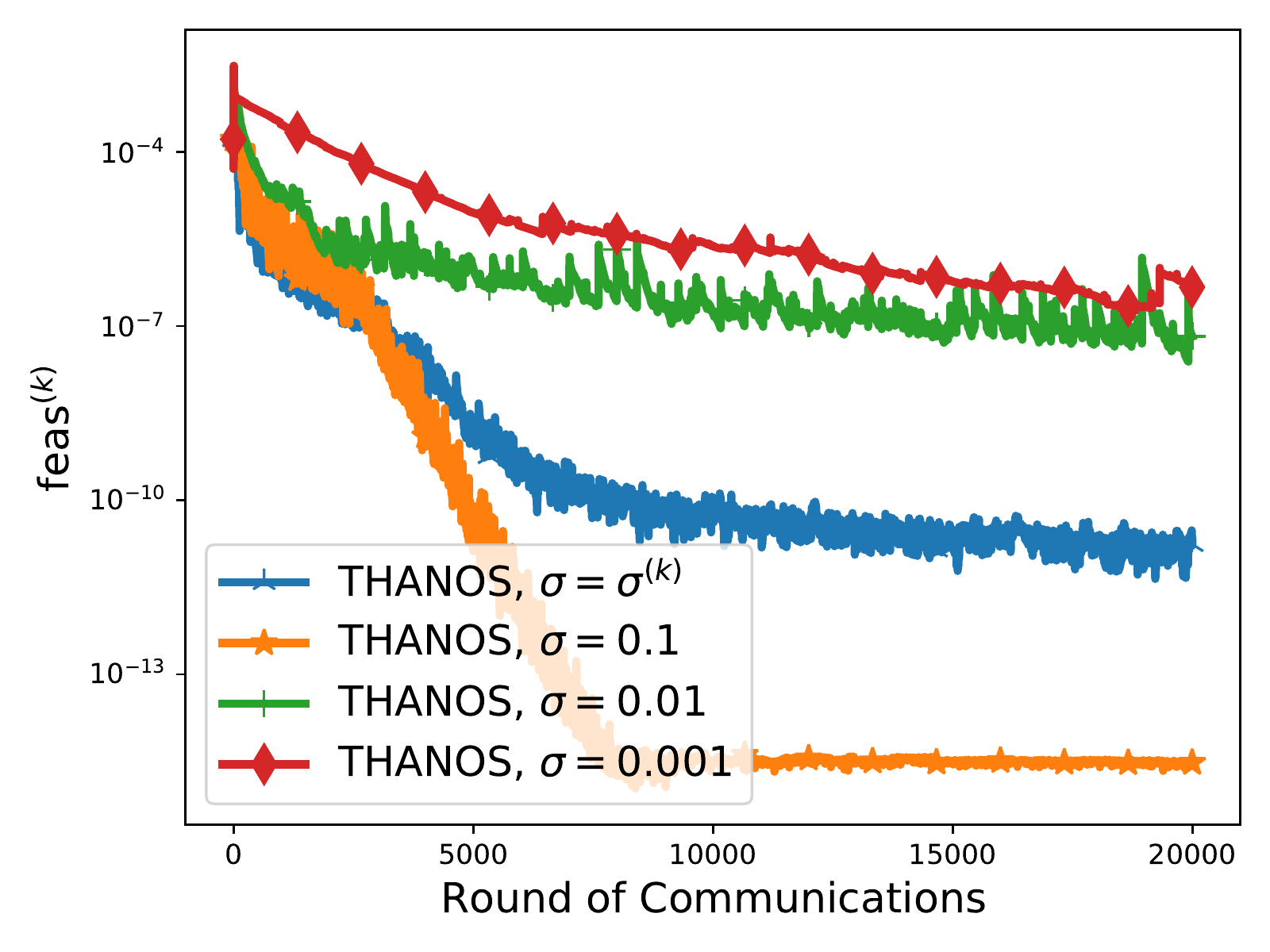}
	}
	
	\caption{Numerical performance of THANOS 
		for different values of $\sigma$
		on sparse PCA problems
		with $r (X) = \norm{X}_1$.}
	\label{fig:SPCA_L1}
\end{figure}

\begin{figure}[t!]
	\centering
	
	\subfigure[$\dist^{(k)}$]{
		\label{subfig:dist_l21}
		\includegraphics[width=0.8\linewidth]{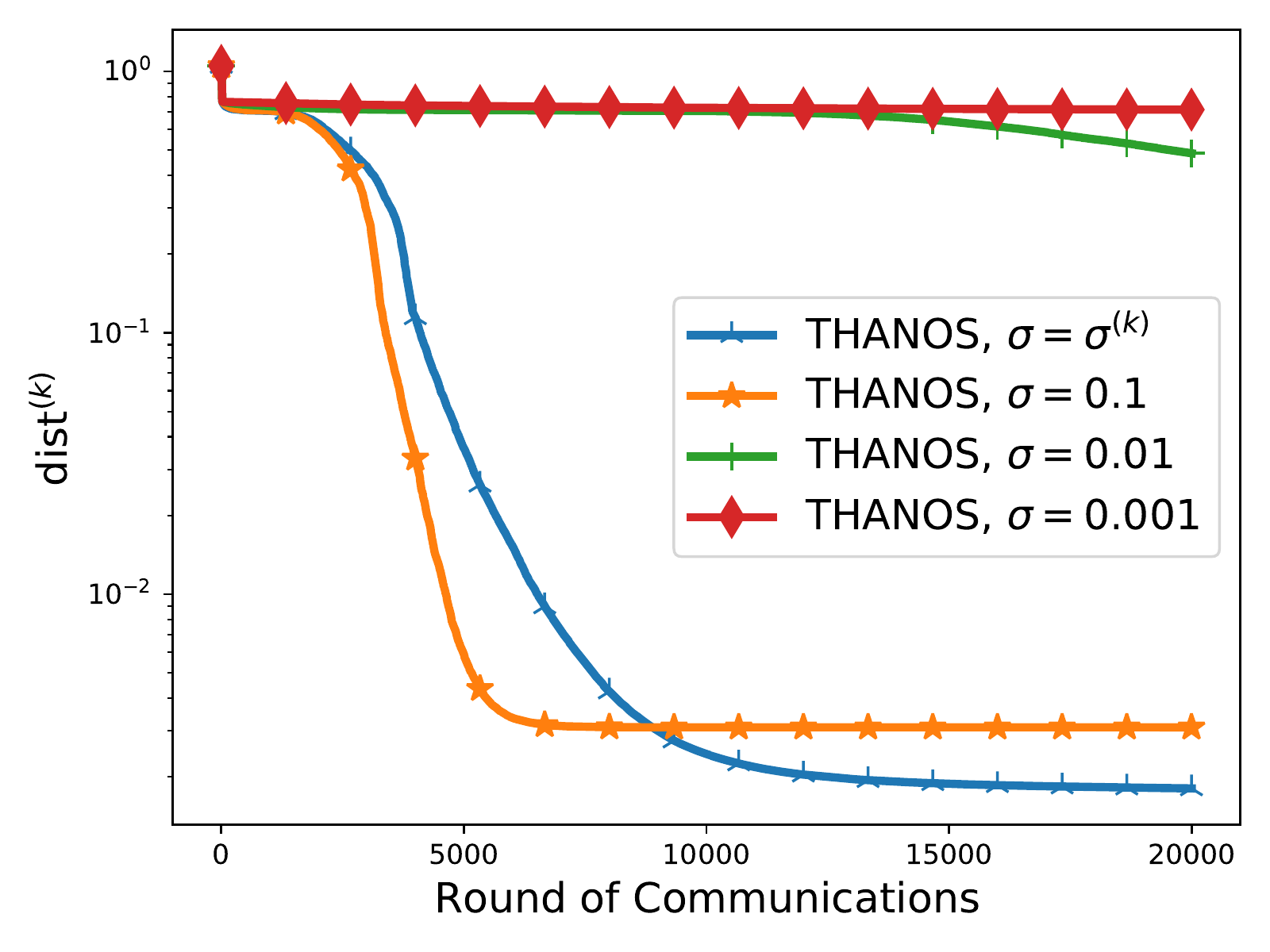}
	}
	
	\subfigure[$\mathrm{feas}^{(k)}$]{
		\label{subfig:feas_l21}
		\includegraphics[width=0.8\linewidth]{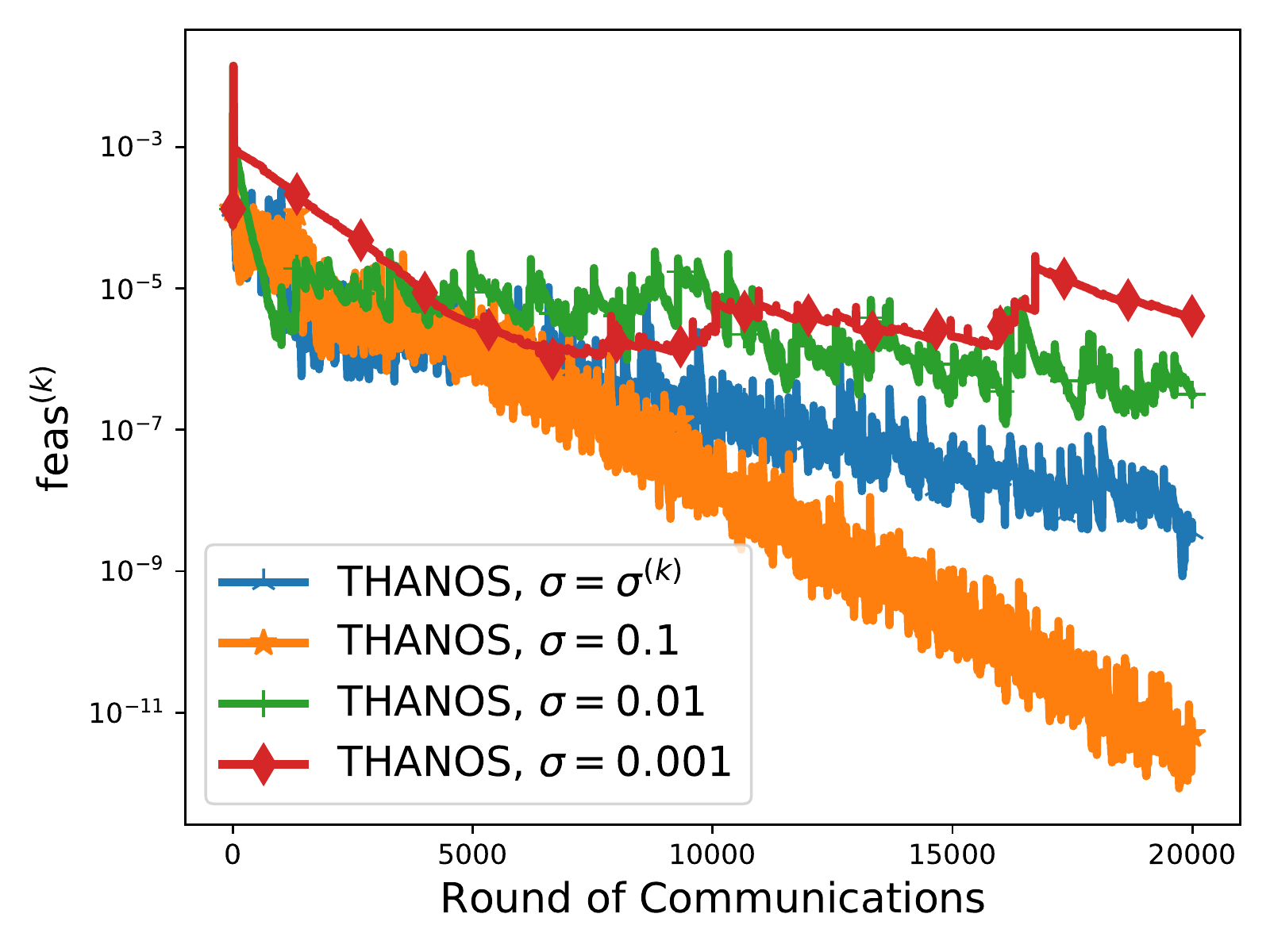}
	}
	
	\caption{Numerical performance of THANOS 
		for different values of $\sigma$
		on sparse PCA problems
		with $r (X) = \norm{X}_{2, 1}$.}
	\label{fig:SPCA_L21}
\end{figure}

Figure \ref{fig:SPCA_L1} and Figure \ref{fig:SPCA_L21} 
depict the numerical performance of THANOS
for two regularizers \eqref{eq:l1} and \eqref{eq:l21}, respectively.
In both figures, we plot $\dist^{(k)}$ and $\mathrm{feas}^{(k)}$
against the iteration count $k$ 
corresponding to different values of $\sigma$,
which are distinguished by colors.
We can observe that, the smaller the value of $\sigma$ is, 
the worse the performance of THANOS becomes.
The reason is that the smoothed problem \eqref{opt:stiefel-s} is ill-conditioned 
for small values of $\sigma$.
Moreover, increasing the value of $\sigma$ will give rise to large approximation errors. 
In order to remedy this dilemma,
we propose an updating scheme 
that gradually reduces the smoothing parameter, 
that is,
\begin{equation*}
	\sigma^{(k)} = k^{-1/3},
\end{equation*}
where $\sigma^{(k)}$ is the smoothing parameter at iteration $k$.
The above updating scheme has a favorable numerical performance in practice,
which is also shown in Figure~\ref{fig:SPCA_L1} and Figure~\ref{fig:SPCA_L21}.

\section{CONCLUSIONS}

This paper considers a class of decentralized optimization problems 
over the Stiefel manifold with non-smooth regularizers.
There is currently no algorithm in the literature
that is capable of solving this problem.
To overcome the difficulty of non-smoothness,
we use the Moreau envelope to approximate the non-smooth regularizers 
in the objective function.
Then we apply an existing algorithm to solve the obtained smooth proxy of the original problem.
The resulting algorithm is called THANOS.
We prove that THANOS will return a first-order $\epsilon$-stationary point
in at most $\cO (\epsilon^{-4})$ iterations.
Preliminary numerical results illustrate that THANOS is of great potential.



\bibliographystyle{ieeetr}

\bibliography{library}

\addcontentsline{toc}{section}{References}

\end{document}